\documentclass[letterpaper, 10 pt, conference]{ieeeconf}  

\IEEEoverridecommandlockouts                              
\usepackage{ntheorem}
\usepackage{hyperref}

\usepackage{textcomp}
\let\labelindent\relax
\usepackage{enumitem}
\usepackage{graphicx}
\usepackage{epstopdf}
\usepackage{float}
\usepackage[nocomma]{optidef}
\usepackage{dsfont}
\usepackage{times}
\usepackage{bm}
\usepackage{fancyhdr,graphicx,amsmath,amssymb}

\usepackage{cleveref}
\crefformat{footnote}{#2\footnotemark[#1]#3}
\usepackage{algorithm} 
\usepackage{algpseudocode} 
\include{pythonlisting}
\usepackage{cite}
\usepackage{xcolor}
\usepackage{amssymb}
\usepackage{mathtools}

\DeclareMathOperator*{\argmax}{arg\,max}

\makeatletter
\def\BState{\State\hskip-\ALG@thistlm}
\makeatother
\usepackage{stfloats}

\title{\LARGE \bf
Safe Dual Gradient Method for Network Utility Maximization Problems
}

\author{Berkay Turan  \and Mahnoosh Alizadeh
\thanks{B. Turan and M. Alizadeh are with  Dept. of ECE, UCSB, Santa Barbara, CA, USA. This work is supported by  NSF grant \#1847096. E-mails: {\url{bturan@ucsb.edu}, \url{alizadeh@ucsb.edu}}}
}

\begin{document}

\maketitle
\thispagestyle{empty}
\pagestyle{empty}

\begin{abstract}
In this paper, we introduce a novel first-order dual gradient algorithm for solving network utility maximization problems that arise in resource allocation schemes over networks with safety-critical constraints. Inspired by applications where customers' demand can only be affected through posted prices and real-time two-way communication with customers is not available, we require an algorithm to generate \textit{safe prices}. This means that at no iteration should the realized demand in response to the posted prices violate the safety constraints of the network.  Thus, in contrast to existing first-order methods, our algorithm, called the safe dual gradient method (SDGM), is guaranteed to produce feasible primal iterates at all iterations. We ensure primal feasibility by 1) adding a diminishing safety margin to the constraints, and 2) using a sign-based dual update method with different step sizes for plus and minus directions. In addition, we prove that the primal iterates produced by the SDGM achieve a sublinear static regret of ${\cal O}(\sqrt{T})$.

\end{abstract}
\theoremseparator{.}

\newtheorem{proposition}{Proposition}
\newtheorem{theorem}{Theorem}
\newtheorem*{theoremnonumber}{Theorem}
\newtheorem{corollary}{Corollary}
\newtheorem{lemma}{Lemma}
\newtheorem{Fact}{Fact}
\newtheorem{remark}{Remark}
\newtheorem{assumption}{Assumption}
\newtheorem{definition}{Definition}

\newcommand{\eqdef}{\vcentcolon=}
\newcommand{\beq}{\begin{equation}}
\newcommand{\eeq}{\end{equation}}
\newcommand{\ie}{i.e., }

\section{INTRODUCTION}

Many applications falling within the scope of resource allocation over networks, e.g., power distribution systems \cite{samadi2010optimal}, congestion control in data networks \cite{kelly1998rate}, and wireless cellular networks \cite{chiang2004balancing}, deal with a multi-agent optimization problem that falls under the general umbrella of \emph{network utility maximization} (NUM) problems. The shared goal in these problems is to allocate the resources to the users subject to a set of coupling constraints such that the total utility of the users is maximized.

In NUM problems, the user-specific utility functions are assumed to be private to the users and therefore a centralized solution is not possible. Accordingly, distributed optimization methods have become suitable tools thanks to the separable structure of NUM problems \cite{palomar2006tutorial,necoara2011parallel}. The idea is to decompose the main problem into sub-problems that can be solved by the individual users and use the solutions of the sub-problems to solve the main problem \cite{bertsekas1997nonlinear,bertsekas2015parallel}, which has been widely advocated for use in different applications, e.g., \cite{li2011optimal,kelly1998rate}. Among the two main types of decomposition methods, primal decomposition methods correspond to a direct allocation of the resources by a central coordinator and solve the primal problem, whereas dual decomposition methods based on the Lagrangian dual problem \cite{shor2012minimization} correspond to resource allocation via pricing and solve the dual problem \cite{palomar2006tutorial}. Due to the structure of NUM problems, the latter approach has been widely adopted in the literature \cite{palomar2006tutorial,nedic2009approximate,beck20141}. Additionally, it gives users the freedom of determining their own resource consumption based on prices.

The (dual) subgradient method is the most basic approach for solving the Lagrangian dual problem, whose convergence properties under various step-size rules have been well established \cite{polyak1987introduction,shor2012minimization,nesterov2009primal,bertsekas1997nonlinear,bertsekas2003convex}. Besides the performance of the subgradient method in the dual space,  scholars have also focused on characterizing the suboptimality, the convergence rate, and the infeasibility amount of the primal iterates\footnote{The suboptimality is measured by the difference in the optimal objective value and the objective value of the iterates. The convergence rate is measured by the distance of the iterates to the optimal solution. The infeasibility amount is measured by the norm of the constraint violation.}. Under no strong concavity or smoothness assumptions on the objective function, the average of the primal sequence produced by the dual subgradient method is shown to achieve an ${\cal O}(1/\sqrt{t})$ primal suboptimality and infeasibility \cite{nedic2009approximate}, where $t$ is the iteration number. When the objective function is strongly concave, the dual problem is smooth and therefore it is possible to achieve a rate of ${\cal O}(1/t)$ for the primal suboptimality, convergence rate, and infeasibility of the last iterate \cite{beck20141} using accelerated methods (e.g., \cite{nesterov1983method}). Using primal averaging schemes, rates of ${\cal O}(1/t^2)$ for the primal suboptimality and infeasibility can be obtained \cite{necoara2013rate,patrinos2013accelerated,chernov2016fast}. Under an additional smoothness assumption on the objective function, global linear convergence rates are achieved for linearly constrained \cite{necoara2015linear} and unconstrained \cite{scaman2017optimal,uribe2020dual} convex optimization problems over networks. It is worthwhile to highlight that none of these works guarantee feasible iterates throughout the optimization process, but only provide bounds on the infeasibility amount of the primal iterates. Therefore, solutions are only realizable after convergence to a near-feasible point for resource allocation systems with safety-critical constraints (even at a near-feasible point, the amount of infeasibility still needs to be accounted for).


This paper is motivated by network resource allocation applications, where 1) users determine their own resource consumption in response to the prices and the realized consumption is only observed afterward, and 2) the system has safety-critical hard constraints that should not be violated by the users' resource consumption at any time. For instance, in price-based demand response, users determine their own electricity consumption in response to prices, where the prices should be set such that the realized demand does not violate the capacity constraints of the electric grid \cite{vardakas2014survey}. This is to ensure the safe operation of the system because violating the capacity constraints could cause physical damage to the grid. This implies that the resource consumption of the users (i.e., primal variables) in response to the prices (i.e., the dual variables) should always satisfy the constraints of the system (i.e., be feasible). This allows users to have realizable demand all the time and gain some utility throughout the optimization process. Unlike existing methods which produce prices that are only implementable (i.e., safe) after convergence to a near-optimal point, our framework does not require convergence before prices can be posted. Hence, it removes the need for complex negotiations with users over what their potential demand would be in response to different prices in order to converge to the optimal price.


To this end, in this paper, we develop a distributed algorithm for NUM based on the dual decomposition scheme, called the safe dual gradient method (SDGM), that produces feasible primal iterates at all iterations. Our method does not use any second-order information (except for a lower bound on the strong concavity constant) and the dual updates solely rely on the constraints evaluated at the current feasible primal iterate.
Our contributions are as follows:
\begin{itemize}[wide]
    \item We introduce a novel algorithm, the SDGM, for solving NUM problems in a distributed fashion. We characterize a principled way to choose algorithm parameters to guarantee feasible primal iterates at all iterations.
    \item We prove that the static regret incurred by the feasible primal iterates produced by the SDGM, i.e., the cumulative gap between the optimal objective value and the objective function evaluated at the primal iterates, up to time $T$ is bounded by ${\cal O}(\sqrt{T})$.
    \item We numerically evaluate our algorithm to support our theoretical findings and compare its performance to existing first-order distributed methods for NUM problems.
\end{itemize}
The primal feasibility and the regret guarantees of the SDGM result from a combination of two ingredients: 1) by adding a safety margin to the constraints, we perturb the dual gradients and increase the dual variables before the constraint is violated (in contrast, the basic dual subgradient method \cite{palomar2006tutorial,beck20141} only increases a dual variable after the corresponding constraint is violated), and 2) we only use the sign of the perturbed gradient and utilize different step-sizes for plus and minus directions. The latter allows us to have global control over the changes in the dual variables independent of the values of the constraints. This is done to ensure a sufficient amount of increase in a dual variable whenever the corresponding constraint is close to being tight, which is crucial for the feasibility of the primal iterates.

\noindent
\textbf{Related work: }
    Besides dual (sub)gradient methods for solving the NUM problem, our work is closely related to interior point methods and safe learning/optimization literature.
\begin{enumerate}[wide]
    \item \emph{Interior point methods: }Interior point methods solve an inequality constrained problem by converting it into a sequence of equality constrained problems using barrier functions, and implementing Newton's method to solve the sequence of problems \cite{boyd2004convex}. They produce feasible iterates, however,  Newton's method is a second-order method that requires the Hessian, which is generally not available in the applications of interest to this paper, such as demand response with no two-way communications. Accordingly, in \cite{armand2000feasible} a feasible interior point method is introduced by approximating the Hessian using the first-order information. However, the algorithm defines a primal update rule, whereas in practical applications we would like to allow users to freely determine their resource consumption in response to the posted prices. Similarly in \cite{wei2010distributed}, the authors propose a distributed Newton method for NUM problems, where the Hessian is approximated by second-order information exchange between the users and the primal updates follow a Newton direction update rule. Closest to the setup we study in this paper would be \cite{athuraliya2000optimization,necoara2009interior}. In \cite{athuraliya2000optimization}, a Newton-like dual update is proposed by approximating the Hessian using only the first-order information. However, only asymptotic convergence of the algorithm is proven and the feasibility of the primal iterates is not guaranteed. On the other hand, the authors of \cite{necoara2009interior} propose an interior point method using Lagrangian dual decomposition with theoretical guarantees, however,  it requires the exact Hessian for the dual update.
    \item \emph{Safe optimization/learning: }This line of work aims to develop safe algorithms that produce feasible iterates/actions for the optimization\cite{usmanova2019safe, allibhoy2021anytime}/bandit\cite{amani2019linear} frameworks, respectively. In \cite{usmanova2019safe} and \cite{amani2019linear}, the feasible set is unknown and the approach is to conservatively estimate the feasible set and pick the primal iterates/actions accordingly. In \cite{allibhoy2021anytime}, the gradient flow that directly optimizes the primal variables is augmented with a control barrier function to maintain safety. In contrast, although the feasible set is known, the dual decomposition architecture does not allow for direct control of the primal iterates, which differentiates our work from this literature.
\end{enumerate}
\noindent
\textbf{Paper Organization:} The remainder of the paper is organized as follows. In Section~\ref{sec:problem}, we formalize the problem setup. In Section~\ref{sec:sdgm}, we describe the SDGM (Algorithm~\ref{alg:safedual}) and in Section~\ref{sec:regret}, we prove its feasibility and regret guarantees. In Section~\ref{sec:num}, we provide a numerical study demonstrating the efficacy of the SDGM.

\noindent
\textbf{Notation and Basic Definitions:} We denote the set of real numbers by ${\mathbb R}$ and the set of non-negative real numbers by ${\mathbb R}_+$. Unless otherwise specified, $\| \cdot \|$ denotes the standard Euclidean norm and $\|\cdot\|_p$ denotes the $p$-norm. Given a positive integer $n>0$, $[n]$ denotes the set of integers $\{1,2,\dots,n\}$. Given a vector $x\in{\mathbb R}^n$, $x_i\in{\mathbb R}$ or $[x]_i\in{\mathbb R}$ denotes the $i$'th entry of $x$. Given a function $f:\mathbb{R}\rightarrow\mathbb{R}$, $f'$ denotes the first derivative of $f$. For a vector $x\in{\mathbb R}^m$, $[x]_+$ is the component-wise maximum of the vector $x$ and the zero vector. Given two vectors $x,y\in {\mathbb R}^m$, $x\leq y$ implies element-wise inequality. The vector $e_m\in{\mathbb R}^m$ denotes the $m$ dimensional vector with all elements equal to 1.

\begin{definition}\label{def:strong}
A differentiable function $f(\cdot)$ is said to be \textbf{$\boldsymbol{\mu}$-strongly concave} over the domain ${\cal X}$ if there exists $\mu>0$ such that
\begin{equation}
    \langle \nabla f(x_2)-\nabla f(x_1),x_1-x_2 \rangle\geq \mu\|x_1-x_2\|^2
\end{equation}
holds for all $x_1,x_2\in \cal X$.
\end{definition}

\begin{definition}\label{def:smooth}
A differentiable function $f(\cdot)$ is said to be \textbf{$\boldsymbol{L}$-smooth} over the domain ${\cal X}$ if there exists $L>0$ such that
\begin{equation}
    \|\nabla f(x_1)-\nabla f(x_2)\|\leq L \|x_1-x_2\|
\end{equation}
holds for all $x_1,x_2\in \cal X$.
\end{definition}

\section{PROBLEM SETUP}\label{sec:problem}
We study the standard NUM problem \cite{kelly1998rate}, where the goal is to allocate resources to $n$ users subject to a set of linear coupling constraints such that the total utility of the users is maximized. It can be formulated as the following optimization problem:
\begin{subequations}\label{eq:main}
\begin{align}
    \underset{x\in{\cal X}\subset\label{eq:objective} \mathbb{R}^n}{\max}&~ f(x)=\sum_{i=1}^n f_i(x_i)\\
    \textnormal{s.t.}&~Ax\leq c,\label{eq:constraint}
\end{align}
\end{subequations}
where $f_i(\cdot)$ is the strictly increasing and concave utility function of user $i$, ${\cal X}={\cal X}_1\times {\cal X}_2\dots \times{\cal X}_n$ with ${\cal X}_i=[\underline{x}_i,\overline{x}_i]$ (where $\underline{x}_i\geq 0$, and $\overline{x}_i=\infty$ is allowed), $c\in \mathbb{R}_+^m$, and $A\in\{0,1\}^{m\times n}$ is an $m\times n$ matrix\footnote{In this work, we study the setting where $A\in\{0,1\}^{m\times n}$ and leave the case where $A\in \mathbb{R}^{m\times n}$ as a future direction.}.

Due to the constraint \eqref{eq:constraint}, the feasible region of \eqref{eq:main} $\overline{{\cal X}}={\cal X}\cap\{x|Ax\leq c\}$ is compact, where $\overline{{\cal X}}\subseteq \prod_{i\in [n]} [\underline{x}_i,\max_{j\in[m]}c_j]$. Over this region ${\overline{\cal X}}$, we make the following assumption on the utility functions:
\begin{assumption}\label{ass:stronglyconcave}
For all $i\in[n]$, the utility function $f_i(\cdot)$ is $\mu_i$-strongly concave over $\overline{\cal{X}}$.
\end{assumption}
Standard utility functions considered in the literature such as the $\alpha$-fair utility functions (see
\cite{mo2000fair}) satisfy the strong concavity assumption over the compact interval ${\overline{\cal{X}}}$. Note that the objective function \eqref{eq:objective} is strongly concave with coefficient $\mu=\min_{i\in[n]}\mu_i$. Accordingly, the convex optimization problem \eqref{eq:main}, whenever feasible, has a unique solution denoted by $x^\star$ and an optimal objective value denoted by $f^\star$. 

\begin{algorithm}[t]
    \caption{Dual Subgradient Method \cite{palomar2006tutorial}}
    \begin{algorithmic}[1]
    \Require Initialize $\lambda^1\geq 0$, step size $\gamma$  
        \For{$t=1,2,\dots$}
        	\State Each user $i\in[n]$ receives $p_i^t\vcentcolon=[A^T\lambda^t]_i$ and solves
			\begin{equation}
			    x_i^t=\underset{x_i\in{\cal X}_i}{\argmax}f_i(x_i)-p_i^tx_i
			\end{equation}
            \State The dual vector $\lambda^t$ is updated as:
            \begin{align}
                \lambda^{t+1}=\max\{0,\lambda^t+\gamma(Ax^t-c)\}
            \end{align}
        \EndFor
    \end{algorithmic}
    \label{alg:dualsub}
\end{algorithm}
Since $f_i(\cdot)$ are private to the users, \eqref{eq:main} can not be solved centrally. Therefore, dual decomposition methods have been proposed in order to decompose \eqref{eq:main} into subproblems that can be solved in a distributed fashion \cite{palomar2006tutorial}. In order to decompose \eqref{eq:main}, we let $\lambda\in \mathbb{R}_+^m$ be the dual vector and form the Lagrangian:
\begin{equation}
    \begin{split}
    q(\lambda)=\underset{x\in{\cal X}}{\max}~ \sum_{i=1}^nf_i(x_i)-\lambda^T(Ax-c)
    \end{split}
\end{equation}
The Lagrangian formulation allows us to separate \eqref{eq:main} into two levels of optimization. At the lower level, each user $i\in[n]$ solves the following subproblem using their own utility function and the pricing signal $p_i\vcentcolon=[\lambda^TA]_i$:
\begin{equation}\label{eq:user}
    q_i(\lambda)=\underset{x_i\in{\cal X}_i}{\max} f_i(x_i)-p_ix_i.
\end{equation}
At the higher level, the master problem determines the dual variable by solving the dual problem:
\begin{equation}\label{eq:dual}
    \underset{\lambda\geq 0}{\min}~ q(\lambda)=\sum_{i=1}^n q_i(\lambda)+\lambda^Tc.
\end{equation}
We note that this approach solves the dual problem instead of \eqref{eq:main}. However, the optimal solutions to both problems coincide under strong duality, which is satisfied under the following assumption (Slater's condition):
\begin{assumption}\label{ass:slater}
There exists a vector $\tilde{x}$ in the relative interior of ${\cal X}$ such that $A\tilde{x}<c$.
\end{assumption}
It is well-known (see \cite{rockafellar2015convex}) that, under Assumption~\ref{ass:slater}, strong duality holds for problem \eqref{eq:main}.

In this work, we measure the performance of an algorithm by the suboptimality, convergence rate, and infeasibility of the primal iterates it produces, which are defined as follows:
\begin{definition}
Let $\{x^t\}$, $t\geq 1$, be a sequence of primal iterates produced by an algorithm. For an iterate $x^t$, we define the suboptimality as $f^\star-f(x^t)$ and the infeasibility as $\|[Ax^t-c]_+\|$. The sequence is said to have a convergence rate of $r(t)$, if $\|x^t-x^\star\|\leq r(t)$ for all $t\geq 1$ and $\underset{t\rightarrow\infty}{\lim}r(t)=0$.
\end{definition}

Since the dual problem is convex, one approach for solving the dual problem \eqref{eq:dual} (and thus also \eqref{eq:main}) is to employ the projected subgradient method with constant step-size $\gamma$ outlined in Algorithm~\ref{alg:dualsub} \cite{palomar2006tutorial,beck20141}. It has been shown that without a strong concavity assumption on $f(\cdot)$, after $t$ iterations the objective function value evaluated at $\overline{x}=(1/t)\sum_{i=1}^t x^i$ achieves $f^\star-f(\overline{x}) \leq {\cal O}(\gamma+1/(t\gamma))$\cite{nedic2009approximate}. Additionally, it has been shown that the primal infeasibility $\|[A\overline{x}-c]_+\|$ is bounded by ${\cal O}(1/(t\gamma))$\cite{nedic2009approximate}. These results however do not consider strongly concave objective functions and, thus, the results there remain within the domain of non-smooth convex optimization (a rate of ${\cal O}(1/\sqrt{t})$ is achieved with $ \gamma={\cal O}(1/\sqrt{t})$). When the objective function is strongly concave, the resulting dual objective is smooth, and therefore it is possible to achieve primal suboptimality, convergence rate, and infeasibility of ${\cal O} (1/t)$ for the last iterate \cite{beck20141} (or primal suboptimality and infeasibility of ${\cal O}(1/t^2)$ for the average iterate \cite{necoara2013rate,patrinos2013accelerated,chernov2016fast}). When the objective function is both smooth and strongly concave, global linear convergence rates can be achieved \cite{necoara2015linear}. 


Although existing distributed first-order methods establish bounds on the infeasibility of the primal iterates, there is no obvious way to modify the algorithms such that the primal iterates are always feasible. In the next section, we propose a first-order algorithm based on the dual decomposition scheme that produces feasible primal solutions at all iterations. In addition, the algorithm should produce primal iterates that result in a sublinear static regret, which is measured by
\begin{equation}\label{eq:regret}
    R(T)=\sum_{t=1}^T f^\star-f(x^t).
\end{equation}
We note that the above definition of regret corresponds to the cumulative sum of suboptimalities of the primal iterates. When the primal iterates are feasible, the solutions are implementable, and therefore regret is a well-defined measure. On the other hand, although the above sum is computable for many of the existing works (e.g., \cite{beck20141,necoara2013rate}), regret is not a well-defined metric since the primal iterates are not necessarily feasible and therefore realizable.

\section{SAFE DUAL GRADIENT METHOD}\label{sec:sdgm}

\begin{algorithm}[t]
    \caption{Safe Dual Gradient Method}
    \begin{algorithmic}[1]
    \Require Initialize $\lambda^1_j=\overline{\lambda}$ for all $j\in[m]$, step sizes $\gamma_-^t$ and $\gamma_+^t$, safety margin vector $\Delta^t\in \mathbb{R}_+^m$   
        \For{$t=1,2,\dots$}
        	\State Each user $i\in[n]$ receives $p_i^t\vcentcolon=[A^T\lambda^t]_i$ and solves
			\begin{equation}
			    x_i^t=\underset{x_i\in{\cal X}_i}{\argmax}f_i(x_i)-p_i^tx_i
			\end{equation}
            \State The dual vector $\lambda^t$ is updated as:
            \begin{align}
                 \label{eq:dualminus}\lambda_j^{t+1}&=\max\{0,\lambda_j^t{-}\gamma_-^t\},~\hspace{-.05cm}\textnormal{if}~[Ax^t+\Delta^t-c]_j{<}0 \\
                  \label{eq:dualplus}\lambda_j^{t+1}&=\min\{\overline{\lambda},\lambda_j^t{+}\gamma_+^t\},~\textnormal{if}~[Ax^t+\Delta^t-c]_j{\geq} 0 
            \end{align}
        \EndFor
    \end{algorithmic}
    \label{alg:safedual}
\end{algorithm}

In this section, we describe the dual update method we propose that produces feasible primal iterates satisfying a sublinear regret. The algorithm, called the safe dual gradient method (SDGM), is outlined in Algorithm~\ref{alg:safedual}. At the heart of the algorithm lie two key ideas: 
\begin{enumerate}[wide]
    \item The classical subgradient method only increases the dual variables after a constraint has been violated, which results in an infeasible primal solution. We add a safety margin ${\Delta}^t$ to the constraints as
    \begin{equation*}
       Ax^t-c\rightarrow Ax^t-c+\Delta^t
    \end{equation*}
    so that the dual variable $\lambda_j^t$ increases when the constraint is $\Delta^t_j$ close to being tight.
    \item In the SDGM, the amounts of the dual updates \eqref{eq:dualminus}-\eqref{eq:dualplus} are independent of the values of the modified constraints (i.e., the perturbed gradient), but only dependent on their signs. This is to ensure that when the $j$'th constraint is $\Delta^t_j$ close to being tight, we increase $\lambda_j^t$ sufficiently by an amount of $\gamma_+^t$ while controlling the reduction of the other dual variables by $\gamma_-^t$ so that the constraint is not violated at iteration $t+1$. If we were to use the actual values of the constraints, then for a constraint $j$ that is ${\Delta^t_j}$ close to being tight, we could only ensure less than ${\cal O}(\Delta^t_j)$ (potentially very close to 0) increase in $\lambda_j^t$. Combined with a reduction in another constraint $k$, which can be as big as ${\cal O}(c_k)$, this might result in a large increase in $x_i^t$ for which $A_{ji}=A_{ki}=1$ (since $p_i^{t}-p_i^{t+1}=[A^T(\lambda^t-\lambda^{t+1})]_i$ can be big). A large increase in $x_i^t$ could cause the constraint $j$ to be violated at the next iteration. On the other hand, by using a normalized update rule we ensure that $x_i^t$ does not increase at the next iteration.
\end{enumerate}

We note that Algorithm~\ref{alg:safedual} is similar to sign gradient methods \cite{riedmiller1993direct}, where the plus and the minus directions have different step sizes. Although convergence guarantees of sign-based gradient methods have been established for unconstrained non-convex optimization \cite{bernstein2018signsgd}, we are not aware of explicit non-asymptotic converge rates for convex optimization with inequality constraints (even using the same step sizes for plus and minus updates).

It is necessary that the initial dual variables produce feasible primal solutions. Since this has to hold before getting any feedback from the users, we make the following assumption:
\begin{assumption}\label{ass:lambdabound}
For all constraints $j\in[m]$, there exists a uniform bound $\overline{\lambda}$ such that if $\lambda_j^t=\overline{\lambda}$ then $[Ax^t-c]_j\leq 0$.
\end{assumption}
Assumption~\ref{ass:lambdabound} is not too restrictive and is satisfied in practice. For instance, if $f_i'(\cdot)$ is bounded by $M$ in $\overline{\cal X}$ for all $i\in[n]$, then $\overline{\lambda}=M$ satisfies the assumption. 

In the next section, we characterize a principled way to choose parameters $\Delta^t$, $\gamma_-^t$, and $\gamma_+^t$ in order to produce feasible primal iterates. Additionally, we prove that the regret incurred by the iterates produced by Algorithm~\ref{alg:safedual} is ${\cal O}(\sqrt{T})$.

\section{FEASIBILITY AND REGRET ANALYSIS}\label{sec:regret}
We will first characterize the choice of algorithm parameters that guarantee primal feasibility at all iterations and then prove the regret of Algorithm~\ref{alg:safedual} under this choice of parameters.
\subsection{Feasibility Analysis}
The following proposition characterizes a principled way to choose the parameters $\Delta^t$ and $\gamma_+^t$ with respect to $\gamma_-^t$ that ensures feasible primal iterates:
\begin{proposition}\label{prop:safety}
Let $\Delta_j^t=\frac{[AA^Te_m]_j}{\mu}\gamma_-^t$ for all $j\in[m]$ and $\gamma_+^t=(m-1)\gamma_-^t$. Then for all $t\geq 1$, the iterates $x^t$ produced by Algorithm~\ref{alg:safedual} are feasible, i.e.,
\begin{equation}
    Ax^t-c\leq0,~\forall t\geq 1.
\end{equation}
\end{proposition}
\begin{proof}
We prove Proposition~\ref{prop:safety} by induction. Suppose that $Ax^t-c\leq 0$ holds at time $t$. We consider the following two cases:
\begin{enumerate}[wide]
    \item Pick a constraint $j$ for which $[Ax^t+\Delta^t-c]_j<0$. If for all users, $x_i^{t+1}\leq x_i^t$, then $[Ax^{t+1}-c]_j\leq0$ holds trivially.  If for a user $i$, if $p_i^{t+1}> f_i'(\underline{x}_i)$, then $x_i^{t+1}=\underline{x}_i$ and $x_i^{t+1}\leq x_i^t$ holds. Furthermore, if $p_i^t<f_i'(\overline{x}_i)$, then $x_i^t=\overline{x}_i$ and $x_i^{t+1}\leq x_i^t$ holds. Therefore, in order to have $x_i^{t+1}>x_i^t$, it is necessary to have $f_i'(x_i^{t+1})\geq p_i^{t+1}$, $f_i'(x_i^{t})\leq p_i^t$, and $p_i^{t+1}\leq p_i^t$. Using strong concavity, we have that:
    \begin{align}
         x_i^{t+1}\leq x_i^t+\frac{ f_i'(x_i^t)-f_i'(x_i^{t+1})}{\mu}&\leq\frac{p_i^t-p_i^{t+1}}{\mu}\\
         &\hspace{-.1cm}=\frac{[A^T(\lambda^t-\lambda^{t+1})]_i}{\mu}.
    \end{align}
    From the update rule, we know that for all $j\in[m]$,    $\lambda_j^t-\lambda_j^{t+1}\leq \gamma_-^t$. Therefore, $x_i^{t+1}\leq x_i^t+\frac{\gamma_-^t[A^Te_m]_i}{\mu}$ holds for all $i\in[n]$. Finally, we write
    \begin{equation}
    \begin{split}
        [Ax^{t+1}-c]_j&=[Ax^t-c]_j+[A(x^{t+1}-x^t)]_j\\
        &\leq -\Delta_j^t+\frac{\gamma_-^t [AA^Te_m]_j}{\mu}= 0,
    \end{split}
    \end{equation}
    which concludes the first case.
    \item Now, pick a constraint $j$ for which $[Ax^t+\Delta^t-c]_j>0$. If $\lambda_j^{t+1}=\overline{\lambda}$, then by definition $[Ax^{t+1}-c]_j\leq 0$ holds. Therefore we assume that $\lambda_{j}^{t+1}<\overline{\lambda}$ and  show that for all users $i$ for which $A_{ji}=1$, $p_i^t-p_i^{t+1}=[A^T\lambda^t]_i- [A^T\lambda^{t+1}]_i\leq 0$ (otherwise if $A_{ji}=0$, $x_i^{t+1}$ does not affect constraint $j$). This implies that $x_i^{t+1}\leq x_i^t$.
    \begin{align}
      [A^T\lambda^t]_i-[A^T\lambda^{t+1}]_i&=\sum_kA_{ki}(\lambda_k^t-\lambda_k^{t+1})\\
      &\leq (m-1)\gamma_-^t-\gamma_+^t=0.
    \end{align}
    Since $[Ax^{t}-c]_j\leq 0$ is given, $[Ax^{t+1}-c]_j\leq 0$ holds when $x_i^{t+1}\leq x_i^t$, which concludes the second case.
\end{enumerate}
We have proven that $Ax^{t+1}-c\leq 0$ holds if $Ax^t-c\leq 0$. Since by definition $Ax^1-c\leq 0$ is ensured, $Ax^t-c\leq 0$ holds for all $t\geq 1$.
\end{proof}
Given that under Proposition~\ref{prop:safety}, $x^t$ for all $t\geq 1$ are feasible and therefore implementable, the static regret \eqref{eq:regret} is a valid choice of performance metric. Next, we prove that the regret of Algorithm~\ref{alg:safedual} is ${\cal O}(\sqrt{T})$.
\subsection{Regret Analysis}
The following theorem establishes an upper bound on the regret incurred by the primal iterates produced by Algorithm~\ref{alg:safedual}:
\begin{theorem}\label{thm:regret}
Let $\gamma_-^t=\gamma/\sqrt{t}$ for some $\gamma>0$ and set $\Delta^t$ and $\gamma_+^t$ as in Proposition~\ref{prop:safety}. Then for all $t\geq 1$, the iterates produced by Algorithm~\ref{alg:safedual} are feasible. Furthermore, the regret $R(T)$ for all $T\geq 1$ satisfies
\begin{equation}\label{eq:regretbound}
    R(T)\leq \frac{\overline{\lambda}^2\|c\|_1\sqrt{T}}{\gamma}+2C\gamma\sqrt{T},
\end{equation}
where
\begin{equation}
    C=\|c\|_1+\frac{\overline{\lambda}m(\|A^Te_m\|^2+\rho(A^TA)(m-1)^2/\mu)}{\mu}
\end{equation}
and $\rho(A^TA)$ is the spectral radius of the matrix $A^TA$.
\end{theorem}
\begin{proof}
According to Proposition~\ref{prop:safety}, $x^t$ is feasible and $Ax^t-c\leq 0$ holds for all $t$. Since $x^t$ is the maximizer of the Lagrangian $f(x)-{\lambda^t}^T(Ax-c)$, strong duality implies that 
\begin{equation}\label{eq:22}
 q(\lambda^t)\vcentcolon=f(x^t)-{\lambda^t}^T(Ax^t-c)\geq f^\star.   
\end{equation}
Accordingly, we can write
\begin{equation}
  R(T)=\sum_{t=1}^T f^\star-f(x_t)\leq \sum_{t=1}^T{\lambda^t}^T(c-Ax^t).
\end{equation}
Since $f_i(\cdot)$ are $\mu$-strongly convex, the dual function denoted as $q(\lambda)$ is $L\vcentcolon=\frac{\rho(A^TA)}{\mu}$-smooth \cite[Lemma II.2]{beck20141}. The descent lemma for smooth functions implies (e.g., \cite{bertsekas1997nonlinear}):
\begin{align}
    q(\lambda^{t+1})\leq q(\lambda^t)+\langle\nabla q(\lambda^t),\lambda^{t+1}-\lambda^t\rangle+\frac{L}{2}\|\lambda^{t+1}-\lambda^t\|^2\label{eq:13}
\end{align}
Next, we decompose the $\langle\nabla q(\lambda^t),\lambda^{t+1}-\lambda^t\rangle$ term into three according to the following sets:
\begin{itemize}\setlength\itemsep{0.5em}
    \item Let ${\cal A}^t_1=\{j\in[m]|\lambda_j^t\geq \gamma_-^t, [Ax^t+\Delta^t-c]_j<0\}$.
    \item Let ${\cal A}^t_2=\{j\in[m]|\lambda_j^t<\gamma_-^t, [Ax^t+\Delta^t-c]_j<0\}$.
    \item Let ${\cal A}^t_3=\{j\in[m]| [Ax^t+\Delta^t-c]_j\geq 0\}$.
\end{itemize}
Noting that $\nabla q(\lambda^t)=c-Ax^t$, we write
\begin{align}
    \begin{split}
     &\langle\nabla q(\lambda^t),\lambda^{t+1}-\lambda^t\rangle=-\sum_{j\in{\cal A}^t_1} [c-Ax^t]_j \gamma_-^t \\
     &+\sum_{j\in{\cal A}^t_2} [c-Ax^t]_j (\lambda_j^{t+1}-\lambda_j^t)+\sum_{j\in {\cal A}_3}[c-Ax^t]_j (\lambda_j^{t+1}-\lambda_j^t)
    \end{split}\\
    &\leq -\sum_{j\in{\cal A}^t_1} [c-Ax^t]_j \gamma_-^t +\sum_{j\in{\cal A}^t_3}\Delta_j^t\gamma_+^t\\
    &\leq -\sum_{j\in{\cal A}^t_1} [c-Ax^t]_j \gamma_-^t +\|\Delta^t\|_1\gamma_+^t\label{eq:15}
\end{align}
where the first inequality follows from the fact that for $j\in{\cal A}^t_2$, $(\lambda_j^{t+1}-\lambda_j^t)\leq 0$, and for $j \in{\cal A}^t_3$, $(c-Ax^t)_j\leq \Delta_j^t$ and $(\lambda_j^{t+1}-\lambda_j^t)\leq \gamma_+^t$. We plug \eqref{eq:15} into \eqref{eq:13}, rearrange, and use $\|\lambda^{t+1}-\lambda^t\|^2\leq m(\gamma_+^t)^2$:
\begin{equation}
    \sum_{j\in{\cal A}^t_1} [c-Ax^t]_j\leq \frac{q(\lambda^t)-q(\lambda^{t+1})}{\gamma_-^t}+\frac{\|\Delta^t\|_1 \gamma_+^t}{\gamma_-^t}+\frac{Lm (\gamma_+^t)^2}{\gamma_-^t}.
\end{equation}
Since $\lambda_j^t\leq \overline{\lambda}$ and $c-Ax^t\geq 0$:
\begin{equation}
    \begin{split}
        &\sum_{j\in{\cal A}^t_1} [c-Ax^t]_j\lambda_j^t\leq \overline{\lambda}\sum_{j\in{\cal A}^t_1} (c-Ax^t)_j\\
        &\leq \overline{\lambda}\left(\frac{q(\lambda^t)-q(\lambda^{t+1})}{\gamma_-^t}+\frac{\|\Delta^t\|_1 \gamma_+^t}{\gamma_-^t}+\frac{Lm (\gamma_+^t)^2}{\gamma_-^t}\right)\label{eq:18}
    \end{split}
\end{equation}
For $j\in{\cal A}^t_2$ by definition:
\begin{equation}\label{eq:19}
    \sum_{j\in{\cal A}^t_2} (c-Ax^t)_j\lambda_j^t\leq \|c\|_1\gamma_-^t
\end{equation}
For $j\in{\cal A}^t_3$ by definition:
\begin{equation}\label{eq:20}
    \sum_{j\in{\cal A}^t_3} (c-Ax^t)_j\lambda_j^t\leq \|\Delta^t\|_1\overline{\lambda}
\end{equation}
We sum \eqref{eq:18}, \eqref{eq:19}, and \eqref{eq:20} and plug $\Delta^t$ and $\gamma_+^t$:
\begin{align}
     \nonumber&{\lambda^t}^T(c-Ax^t)\leq \frac{\overline{\lambda}(q(\lambda^t)-q(\lambda^{t+1}))}{\gamma_-^t}+\frac{\overline{\lambda}\gamma_-^t \|A^Te_m\|^2(m-1)}{\mu}\\
     &+\frac{\overline{\lambda}Lm(m-1)^2\gamma_-^t}{\mu}+\|c\|_1\gamma_-^t+\frac{\overline{\lambda} \|A^Te_m\|^2\gamma_-^t}{\mu}
\end{align}
\begin{figure*}[t]
    \centering
    \includegraphics[width=\linewidth]{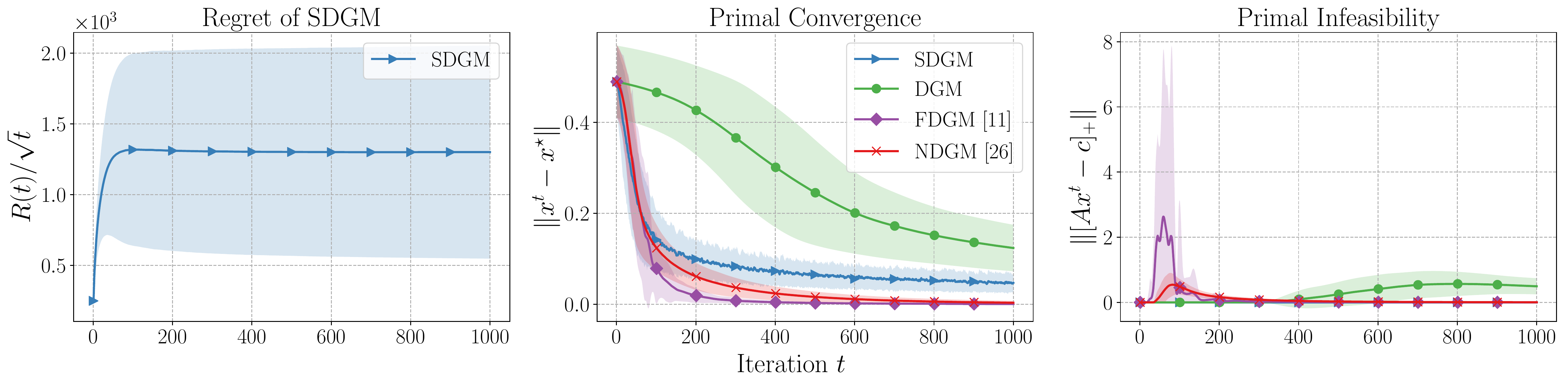}
    \vspace{-.5cm}
    \caption{Results comparing the performances of all four methods for the numerical study described in Section~\ref{sec:num}: Regret of the SDGM (left), the convergence rate of the primal iterates to the optimal solution for all four methods (middle), and the infeasibility amount of the primal iterates for all four methods (right). The solid lines correspond to the means of the 100 experiments, while shaded areas correspond to the standard deviations.}
    \label{fig:experiments}
\end{figure*}
Summing the above inequality from $t=1$ to $T$ with definitions of $C$ and $\gamma_-^t=\gamma/\sqrt{t}$ (with $\gamma_-^0\vcentcolon=0$):
\begin{align}
    &R(T)\leq\overline{\lambda}\sum_{t=1}^T q(\lambda^t) \left(\frac{1}{\gamma_-^t}{-}\frac{1}{\gamma_-^{t-1}}\right){-}\frac{\overline{\lambda}q(\lambda^{T+1})}{\gamma_-^T}{+}C\sum_{t=1}^T \gamma_-^t\\
    &{\leq} \overline{\lambda}(f^\star{+}\overline{\lambda}\|c\|_1)\sum_{t=1}^T \left(\frac{1}{\gamma_-^t}{-}\frac{1}{\gamma_-^{t-1}}\right){-}\frac{\overline{\lambda}f^\star}{\gamma_-^T}{+}C\sum_{t=1}^T \gamma_-^t\\
    &\leq \overline{\lambda}(f^\star+\overline{\lambda}\|c\|_1)\frac{1}{\gamma_-^T}-\frac{\overline{\lambda}f^\star}{\gamma_-^T}+C\sum_{t=1}^T \gamma_-^t\\
    &\leq \overline{\lambda}^2\|c\|_1\frac{\sqrt{T}}{\gamma}+2C\gamma\sqrt{T},
\end{align}
where the first inequality uses \eqref{eq:22} and
\begin{equation}
    q(\lambda^t)=f(x^t)-{\lambda^t}^T(c-Ax^t)\leq f^\star+\overline{\lambda}\|c\|_1
\end{equation}
for any feasible $x^t$.
\end{proof}
According to Theorem~\ref{thm:regret}, Algorithm~\ref{alg:safedual} produces feasible solutions for all $t\geq 1$ that achieve a sublinear regret of ${\cal O}(\sqrt{T})$. To minimize the upper bound on the RHS of \eqref{eq:regretbound}, we let $\gamma=\sqrt{\overline{\lambda}^2\|c\|_1/(2C)}$ to get
\begin{equation}
    R(T)\leq 2\overline{\lambda}\sqrt{2C\|c\|_1T}.
\end{equation}
It is worthwhile to highlight a trade-off between feasibility and performance: It is known that with strongly concave utility functions $f_i(\cdot)$, it is possible to produce primal iterates such that the last iterate (or the average iterate) has suboptimality of ${\cal O}(1/T)$\cite{beck20141} (or ${\cal O}(1/T^2)$\cite{necoara2013rate}) after $T$ iterations, where the primal iterates are not necessarily feasible for all $t\geq 1$. On the other hand, a regret of ${\cal O}(\sqrt{T})$ implies a suboptimality of ${\cal O}(1/\sqrt{T})$ (of the average iterate $\overline{x}$), while producing feasible iterates for all $t\geq 1$. Accordingly, the feasibility of the primal iterates comes at the cost of a slower reduction rate of suboptimality.

In the next section, we numerically demonstrate the feasibility and the convergence rate of, and the regret incurred by the primal iterates produced by Algorithm~\ref{alg:safedual}.

\section{NUMERICAL STUDY}\label{sec:num}

In this section, we compare the performance of the safe dual gradient method developed in Section~\ref{sec:sdgm} with three other distributed algorithms commonly used in the literature for solving the NUM problem: 1) the dual (sub)gradient method (DGM) explained in Section~\ref{sec:problem}, the fast weighted dual gradient method (FDGM) \cite{beck20141}, and  3) the Newton-type diagonally scaled (dual) gradient method (NDGM) introduced in \cite{athuraliya2000optimization}. 

Inspired by \cite{beck20141}, we have implemented all algorithms on a randomly generated collection of 100 networks with a random number of users $n$ taking (integer) values in range $[10,40]$, and a random
number of constraints $m$ taking values in the interval $[5,25]$ (generated independently). For each configuration, we randomly generated the matrix $A$ by sampling $m\times n$ Bernoulli random variables (when a row or a column of $A$ is zero vector, we generate another one). For all users $i\in[n]$, we let the utility function be $\theta_i \log{(x_i+0.1)}$, where $\theta_i$ is sampled from the range $[10,30]$ uniformly at random for each network configuration. We added $0.1$ to the $\log$ function to prevent numerical instability when $x_i$ is close to $0$. We let ${\cal X}_i=[0,\infty)$ for all $i\in[n]$. Finally, we assumed that for all the constraints $j\in[m]$, $c_j=1$. For each configuration, we ran all four methods for $T=1000$  and demonstrate the results in Figure~\ref{fig:experiments}.

In the left figure, we plot the regret incurred by the SDGM. We note that because the other three algorithms do not guarantee primal feasibility for all iterations, regret is not a well-defined metric for their performances. We show that the mean of $R(t)/\sqrt{t}$ is bounded by ${\cal O}(1)$, which implies that $R(t)$ is bounded by ${\cal O}(\sqrt{t})$.

In the middle figure, we plot the convergence of the primal iterates, i.e., $\|x^t-x^\star\|$, for all four methods. The figure shows that although the SDGM is not a fast method as the FDGM and the NDGM, it still performs closely to those fast algorithms and is much better than DGM in terms of convergence rate.

In the right figure, we plot the primal infeasibility, i.e., $\|[Ax^t-c]_+\|$, for all four methods. The figure shows that although FDGM and NDGM are fast, they do not produce feasible iterates for all $t\geq 1$. On the other hand, the SDGM is guaranteed to produce feasible primal iterates.

\section{CONCLUSIONS}
In this work, we introduced a novel algorithm, called the safe dual gradient method (SDGM), for solving NUM problems in a distributed fashion. In contrast to the literature on first-order distributed methods, where bounds on the feasibility violation of the primal iterates are established, the SDGM is guaranteed to produce feasible primal iterates. This is done by: 1) adding a diminishing safety margin to the constraints, and 2) using a sign-based dual update method with different step sizes for plus and minus directions. Furthermore, we have proven that the regret incurred by the SDGM is ${\cal O}(\sqrt{T})$.

An immediate trade-off is that although the SDGM produces feasible iterates, it converges slower than the state-of-the-art methods. It would be an interesting future direction to study an accelerated version of the SDGM (e.g., similar to \cite{nesterov1983method}) that still produces feasible iterates. Additionally, future work should include more general constraints (e.g., $A\in \mathbb{R}^{m\times n}$) as well as the case when there is uncertainty about constraints.



\bibliographystyle{IEEEtran}
\bibliography{references}

\end{document}